\documentclass[a4paper,11pt]{article}
\usepackage{geometry}
\usepackage{amsmath}
\usepackage{enumerate}
\usepackage{amsfonts}
\usepackage{graphicx}
\usepackage{amsthm}

\usepackage[latin1]{inputenc}
\usepackage[T1]{fontenc}
\usepackage[english]{babel}

\usepackage[nottoc, notlof, notlot]{tocbibind}

\usepackage{amsmath}
\usepackage{amssymb}
\usepackage{amsthm}

\newcommand{\corps}{\text{Tr }(\Gamma)}
\newcommand{\Q}{\mathbb{Q}}
\newcommand{\R}{\mathbb{R}}
\newcommand{\C}{\mathbb{C}}
\DeclareMathOperator{\IM}{Im}
\DeclareMathOperator{\RE}{Re}

\newcommand{\Tr}{\text{Tr}}
\newcommand{\ov}{\overline}

\newtheoremstyle{break}
 {\topsep}
 {\topsep}
 {\itshape}
 {}
 {\bfseries}
 {.}
 {\newline}
 {\thmname{#1}\thmnumber{ #2}\thmnote{. \textit{#3}}}

\theoremstyle{plain}
\newtheorem{thm}{Theorem}[section]
\newtheorem{prop}[thm]{Proposition}
\newtheorem{lemma}[thm]{Lemma}

\theoremstyle{definition}

\newtheorem{rem}[thm]{Remark}
\newtheorem*{ex}{Example}
\newtheorem{defn}[thm]{Definition}

\title{Trace fields of subgroups of SU($n$,1)}

\author{Juliette Genzmer\thanks{\texttt{genzmer@math.jussieu.fr}, Institut mathématique de Jussieu, 175 rue du Chevaleret, 75\,013 Paris.}}

\begin{document}


\maketitle

\begin{abstract}
In this note, we study the field generated by the traces of subgroups of SU($n$,1). Under some hypotheses, the trace field of a group $\Gamma \subset$ SU(2,1) is equal to the field generated by the coefficients of the matrices in $\Gamma$. If the group is the image of a representation of the fundamental group of a triangulated 3-manifold, we can relate the trace field to a geometric invariant. For an arithmetic group of the first type in SU($n$,1), up to conjugacy, the trace field and the field of the coefficients are the same.

\end{abstract}

\date{}

\section*{Introduction}
Arithmetic subgroups of PSL$_2(\C)$ are constructed from a pair ($k$,$\mathcal{A})$ where $k$ is a
number field with exactly one complex place and $\mathcal{A}$ is a quaternion algebra (see \cite{Neumann}). Conversely, from an arithmetic group $\Gamma$ one can rebuild these two objects. For
example, the field $k$ is generated by the traces of $\Gamma^{(2)}$ which is the subgroup of $\Gamma$ generated by all squares and is a commensurability invariant. Neumann and Reid extend these definitions to Kleinian groups. They develop the information $k$ gives on the groups. One of these properties is about the non-uniform lattices $\Gamma \subset $ PSL$_2(\C)$. \cite{Penner} proves that $\mathbb{H}_{\R}^3/ \Gamma$ can be triangulated by ideal hyperbolic tetrahedra  (with possibly flat
tetraedra). Each tetrahedron is characterized by a PSL$_2(\C)$-invariant $z\in \C$. It is the cross-ratio of the four vertices. The field generated by these invariants is exactly $k$. To prove this, the authors use the field $\Q(\Tr(\Gamma))$ generated by all traces. It is not a commensurability invariant as shown in \cite{Reid_note}. They show that for such a group $\Gamma$, the coefficients of the matrices are in $\Q(\Tr(\Gamma))$. In fact, this is true for groups that contain a parabolic element and do not fix a point. The proof works even if $\Gamma$ is not discrete.

In this note, we focus on adapting these tools to the complex case. In \cite{McReynolds}, the author extends the definition of $k$ and $\mathcal{A}$ to groups in SU($n$,1) and prove that they are commensurability invariants. We prove that $\Gamma$ is in SU(2,1,$\Q(\Tr(\Gamma))$), after possibly conjugating, if $\Gamma$ contains a parabolic element and is Zariski dense. In the real case, these hypotheses are realized when $\mathbb{H}_{\R}^3/ \Gamma$ is a non-compact manifold of finite volume. For a representation of the fundamental group of a triangulated 3-manifold, the field generated by the invariants of the tetrahedra is $k$ if the holonomy around each vertex is parabolic. These results do not use discreteness. In the case of lattices, we prove that the trace field and the field generated by the coefficients of an arithmetic group of the first type in SU($n$,1) are equal.

We thank N. Bergeron, E. Falbel, B. McReynolds, J. Paupert, E. Ramos and P. Will for many fruitful discussion.

 \section{Complex hyperbolic space}
\label{section1}

We review some material on complex hyperbolic space (for more details, we refer the reader to \cite{Goldman}).

We denote $\C^{n,1}$ the $n+1$-dimensional complex space equipped with a Hermitian form $\left\langle ,\right\rangle $ of signature $(n,1)$, linear in the first variable and antilinear in the second. For $z\in \C^{n,1}$, $\left\langle z,z\right\rangle $ is real, so we can define the following subspaces of $\C^{n,1}$:
 \[V_+=\{z\in \C^{n,1} : \left\langle z,z\right\rangle >0\}, \quad
V_0=\{z\in \C^{n,1} : \left\langle z,z\right\rangle =0\}, \quad
V_-=\{z\in \C^{n,1} : \left\langle z,z\right\rangle <0\}. 
\]
Consider the projection into the projective space $\C P^n$, noted $P$. \textit{Complex hyperbolic $n$-space} is the image under $P$ of $V_-$ equipped with the distance function given by: 
\[\cosh^2\left(\frac{1}{2}d(q_1,q_2) \right) =\dfrac{\left\langle \tilde{q}_1,\tilde{q}_2\right\rangle \left\langle \tilde{q}_2,\tilde{q}_1\right\rangle }{\left\langle \tilde{q}_1,\tilde{q}_1\right\rangle \left\langle \tilde{q}_2,\tilde{q}_2\right\rangle }\]where $\tilde{q}_1$ and $\tilde{q}_2$ are in $\C^{n,1}$, representing $q_1$ and $q_2$ in $\mathbb{H}^n_{\C}$ respectively ($d$ does not depend on the choice of the lifts). This distance corresponds to the Bergmann metric of the unit ball. The function $d$ is normalized so that the holomorphic sectional curvature is $-1$ while the real sectional curvature is pinched between $-1$
and $-1/4$. The boundary of $\mathbb{H}^n_{\C}$ is the projection $P(V_0)$.

Complex conjugation and PU($n$,1) act clearly as isometries on the hyperbolic space. In fact they generate all isometries.

\bigskip 

Each Hermitian form of signature $(n,1)$ gives a model of the hyperbolic space and we can pass between them using a \textit{Cayley transform}. In the case of dimension 2, we will use essentially the Hermitian form associated to the following matrix:
\[ J=\left(
\begin{array}{ c c c }
0&  0  &1 \\
0&1   & 0\\
1&  0  &0
  \end{array} \right).\]
That is $\left\langle z,w\right\rangle =z_1\ov{w_3}+z_2\ov{w_2}+z_3\ov{w_1}$ for $z$ and $w$ in $\C^3$. With this choice of Hermitian form, one obtains the \textit{Siegel domain} model of $\mathbb{H}^2_{\C}$.

\noindent The normalized lift of a point $p\in \mathbb{H}^2_{\C}$ is: 
\[\tilde{p}=\left(\begin{array}{ c  }
-\frac{1}{2}(|z|^2+u-it) \\
z\\
1  \end{array} \right) \quad\text{with }z\in \C,\: t\in \R , \: u \in \:]0,\infty[.\]

\noindent We call $(z,t,u)$ \textit{horospherical coordinates}. To compactify the Siegel domain, we add a point at infinity, denoted $\infty$. One of its lift is $^t(1,0,0)$. The other points of the boundary of this domain correspond to the coordinates $u=0$. So the boundary of the Siegel domain may be identified with $(\C \times \R )\cup \{\infty\}$. 

Another Hermitian form of signature (2,1) is given by: $\left\langle z,z\right\rangle _2=|z_1|^2+|z_2|^2-|z_3|^2$. In the projective space, using the section $z_3=1$, the space $\mathbb{H}^2_{\C}$ is identified with the unit ball and its boundary is $S^3$. 

\bigskip

As the sectional curvature is not constant, there are no totally geodesic real hypersurfaces. The totally geodesic submanifolds of dimension two are either totally real or complex linear subspaces. The first ones are called \textit{$\R$-planes}. They are $P(V) \cap \mathbb{H}^2_\C$ where $V$ is a 3-dimensional subspace of $\C^{2,1}$ such that $\left\langle z,w\right\rangle  \in \R$ for all $z$ and $w$ in $V$. For the second one, consider $c \in V_+$. We define the 2-dimensional subspace $c^\perp=\{z\in \C^{2,1} : \left\langle z,c\right\rangle =0\}$. The signature of the Hermitian form restricted to $c^\perp$ is $(1,1)$ so the intersection with $V_-$ is non empty. 
\begin{defn}
The submanifolds given by $P(c^\perp) \cap \mathbb{H}^2_\C$ are called \textit{complex lines} and the vector $c$ is a \textit{polar vector}.
\end{defn}

Holomorphic isometries of $\mathbb{H}^2_{\C}$ are classified by their fixed points. An element $\gamma$ of PU(2,1) is \textit{parabolic} if it fixes exactly one point in $\partial \mathbb{H}^2_{\C}$. If it fixes exactly two points on the boundary, $\gamma$ is a \textit{loxodromic} transformation. And if $\gamma$ fixes at least one point in $\mathbb{H}^2_{\C}$, it is \textit{elliptic}. 

A lift in SU(2,1) of a parabolic transformation that fixes $\infty$ is:

\[P=e^{i\theta}
 \begin{pmatrix}
1 & -\ov{z}e^{-3i\theta} & -\frac{1}{2}(|z|^2-it) \\
0 & e^{-3i\theta}  & z \\
0 & 0 & 1               
\end{pmatrix}.\]
\noindent \textit{Pure parabolics} are the unipotent one, that is $P-\text{Id}$ is nilpotent. There are two conjugacy classes divided according to the nilpotency index. The first one is represented by $P$ with $\theta=0$, $z=0$ and $t=1$. It acts on the boundary as $(w,s)\mapsto(w,s+1)$, as a vertical translation. The other class is represented by $P$ when $\theta=0$,$z=1$ and $t=1$. It is a horizontal translation.
The \textit{screw parabolics} have two different eigenvalues. Each class of conjugacy is represented by $P$ with $z=0$, $t=1$ and $\theta \neq 0$. They act as a translation on a complex line and as a rotation around this one.

\noindent Up to conjugation, a loxodromic transformation fixes $0$ and $\infty$. A lift in SU(2,1) is:
$\begin{pmatrix}
 \lambda &0 &0 \\
0 & \frac{\ov{\lambda}}{\lambda} &0 \\
0 & 0 & \frac{1}{\;\ov{\lambda}\;}
\end{pmatrix}$. On the boundary, it acts as a dilation: $(z,t)\mapsto\left(\frac{\ov{\lambda}^2}{\lambda}z,|\lambda|^2 t\right)$.

 \section{Trace field}
\begin{defn}
For any subgroup $G$ in $\text{SU}(2,1)$, let us introduce the notation, the field $\Tr(G)$ and the algebra $A(G)$ as follow:
$$\text{Tr}(G)=\mathbb{Q}\left( \text{Tr}(g):g\in G \right) ,$$
$$A(G)=\left\lbrace \sum_i a_i g_i \: :\: a_i\in \Tr(G) , \: g_i \in G\right\rbrace. $$

Now, let $\Gamma$ be a subgroup of PU(2,1). A element $\gamma\in \Gamma$ has three lifts in SU(2,1): $\tilde{\gamma}, \omega\tilde{\gamma}, \omega^2\tilde{\gamma}$ (where $\omega$ is a cube root of unity). We denote $\widetilde{\Gamma}$ the group of all lifts in $\text{SU}(2,1)$ of elements of $\Gamma$. It satisfies the short exact sequence $0\rightarrow \mathbb{Z}_3 \rightarrow\widetilde{\Gamma} \rightarrow \Gamma \rightarrow 1$. The field $\Tr\: \tilde\Gamma$ is stable by complex conjugation $\left(\text{Tr}(\gamma^{-1})=\overline{\text{Tr}(\gamma)}\right)$. 

We will be interested in Tr$(\widetilde{\Gamma}^{(3)})$ where $\widetilde{\Gamma}^{(3)}$ is the subgroup generated by all cubes. The field will be noted $k(\Gamma)$.

\end{defn}

In \cite{McReynolds}, it is proved that $k(\Gamma)$ is a commensurability invariant for $\Gamma$ finitely generated. For the reader's convenience, we repeat here the proof.

\begin{prop}
\label{preuve de Mc}
Let $\Gamma$ be a finitely generated subgroup of PU(2,1). The field $k(\Gamma)$ is a commensurability invariant.
\end{prop}
\begin{proof}
We start by proving that for any finite index subgroup $\Gamma_1 \subset \Gamma$ we have $\Tr(\Gamma^{(3)})\subset \Tr(\Gamma_1)$. Then if $\Gamma$ and $\Delta$ are commensurable, the intersection $\Gamma^{(3)} \cap \Delta^{(3)}$ has finite index in both of them. So $\Tr (\Gamma^{(3)})\subset \Tr(\Gamma^{(3)} \cap \Delta^{(3)})\subset \Tr(\Delta^{(3)})$ and we conclude $k(\Gamma)=k(\Delta)$.

We may assume that $\Gamma_1 $ is normal in $\Gamma$ (otherwise, take the intersection of all conjugates of $\Gamma_1$ under $\Gamma$). Let $g$ be in $\Gamma$. Conjugation by $g$ induces an automorphism of $\Gamma_1$ and so an automorphism $\phi_g$ of $A(\Gamma_1)$.  The set $A(\Gamma_1)$ is a central simple algebra over $\text{Tr}(\Gamma_1)$.
The Skolem-Noether theorem implies that $\phi_g$ is an inner automorphism:

\[ \exists g_1\in (A( \Gamma_1))^\times  \:\:\forall x \in A(\Gamma_1 )\hspace{0.5cm} \phi_g(x)=g_1 x g_1^{-1}.\]

\noindent In $A( \Gamma_1)\otimes \mathbb{C}\cong M(3,\mathbb{C})$, $g_1^{-1}g$ is central, and so: $g_1^{-1}g=\beta \text{Id}$ with $\beta$ in $\mathbb{C}^*$. As $\text{det}(g)=1$, $\beta^3=\text{det}(g_1^{-1})\in \Tr(\Gamma_1)$.

So $g^3=\beta^3 g_1^3\in A(\Gamma_1 )$. This implies the inclusion: $k(\Gamma)=\Tr(\Gamma^{(3)})\subset \Tr(\Gamma_1)$ and concludes the proof.
\end{proof}

\begin{ex} Let $\mathcal{O}_d$ be the ring of integers in the imaginary quadratic number field $\mathbb{Q}(i\sqrt{d})$, where $d$ is a positive square-free integer. If $\Gamma$ is the subgroup of SU(2,1) with coefficients in $\mathcal{O}_d$, then $\Tr({\Gamma}^{(3)})=\mathbb{Q}(i\sqrt{d})$.
\end{ex}

\subsection{Trace field and the coefficients of $\Gamma$}

The next theorem relates the trace field $\Tr \:\Gamma$ and the coefficients of the matrices of ${\Gamma}$ :

\begin{thm}
\label{prop coeff trace}
Let $\Gamma$ be a subgroup of $\text{SU}(2,1)$. Suppose that $\Gamma$ contains a parabolic element $P$. If $\Gamma$ is Zariski dense then it can be conjugated to a subgroup of $\text{SU}(2,1,\text{Tr}({\Gamma})$).
\end{thm}

\noindent We start with the following lemma :
\begin{lemma}
\label{lemma eigenvalues}
Let $P$ be a non-unipotent parabolic matrix. Then, its eigenvalues are in the field generated by Tr$P$ and $\ov{\text{Tr}P}$.
\end{lemma}
\begin{proof}
The matrix $P$ is conjugated to $e^{i\theta}\begin{pmatrix}
1&0&\frac{i}{2} \\
0&e^{-3i\theta}&0 \\
0&0&1
\end{pmatrix}$. The trace of $P$ is: 
\[\text{Tr}(P)=2\cos^2\theta +2\cos\theta-1+2i\sin \theta(1-\cos\theta)\qquad \text{and}\qquad|\text{Tr}(P)|^2=5+4\cos3\theta.\]

\noindent The following formulas prove the result:
\[\cos\theta=\dfrac{\frac{|\text{Tr}(P)|^2-5}{4}+2\RE(\text{Tr}P)+2}{2\RE(\text{Tr}P)+3}\quad\text{ and }\quad i\sin\theta=\dfrac{i\IM(\text{Tr}P)}{2(1-\cos\theta)}.\]
\end{proof}

\noindent Now, we prove the theorem:

\begin{proof}
We will start finding eight matrices $A_i$ in ${\Gamma}$ such that $(\text{Id},A_1,..,A_8)$ is a basis of $M_9(\mathbb{C})$ and its coefficients will be in Tr$({\Gamma})$. Then if $X$ is a matrix in ${\Gamma}$, then its coefficients are solutions of the following system:
\[\left\lbrace \begin{array}{l}
 \text{Tr}(X)=a_0
\\\text{Tr}(XA_1)=a_1
\\...
\\\text{Tr}(XA_8)=a_8
\end{array}\right.\]
with $a_i$ in Tr$({\Gamma})$. This system is linear invertible and its coefficients are in Tr$({\Gamma})$. So, $X \in$ SU(2,1,Tr$({\Gamma})$).

To construct the eight matrices, we use the matrix parabolic $P\in \Gamma$. As the group is non-elementary, we can find $L$ that does not fix the fixed point of $P$ and such that $P$ and $L$ do not stabilize a complex line. 
The aim of the end of the proof is to show that the coefficients of $P$ and $L$ are in Tr$({\Gamma})$. We may assume that $P$ fixes infinity and that $L$ maps infinity to 0. Now, to keep these properties of $P$ and $L$, we are allowed to conjugate $\Gamma$ only by dilations: $ \begin{pmatrix}
\xi&0&0 \\
0&a&0 \\
0&0&\frac{1}{\;\overline{\xi}\;}
\end{pmatrix}$. 
Up to dilations, the conjugacy classes of parabolic transformations that fix infinity are represented by one of the following matrices:  
 
\[ e^{i\theta}\begin{pmatrix}
1&0&\frac{i}{2} \\
0&e^{-3i\theta}&0 \\
0&0&1
\end{pmatrix} \quad \text{ or } \quad   
 e^{i\theta}\begin{pmatrix}
1&-e^{3i\theta}&\frac{-1+is}{2} \\
0&e^{-3i\theta}&1 \\
0&0&1
\end{pmatrix}\quad  \text{(with }s \text{ real}).\]

\noindent All we know about $L$ is that it maps $\infty$ to 0. So, the form of $L$ is:
\[L= \begin{pmatrix}
0&0& \frac{1}{\;\overline{\lambda}\;}\\
0&-\frac{\overline{\lambda}}{\lambda}&\frac{1}{\;\overline{\lambda }\;}\xi\\
\lambda&\frac{\overline{\lambda}}{\lambda}\xi&\frac{1}{2}\frac{-|\xi|^2+iu}{\overline{\lambda}}
\end{pmatrix}.\] 
\noindent As $L$ does not fix a common complex line with $P$, $\xi\neq 0$.

\noindent For each conjugacy class, we prove that $P$ and $L$ are in SU(2,1,Tr$({\Gamma})$).

\bigskip

\noindent \textbf{First case:}

\bigskip

\noindent $P=e^{i\theta}\begin{pmatrix}
1&0&\frac{i}{2} \\
0&e^{-3i\theta}&0 \\
0&0&1
\end{pmatrix}.$
\\To eliminate the value $\xi$ in the matrices, we conjugate the group by $Y=\text{diag}(
\frac{1}{\;\overline{\xi}\;},i,\xi)$. The elements become:
$$P= e^{i\theta}\begin{pmatrix}
1&0&\frac{i}{2|\xi|^2} \\
0&e^{-3i\theta}&0 \\
0&0&1
\end{pmatrix}\quad\text{ and }\quad L= \begin{pmatrix}
0&0&\frac{1}{\overline{\lambda}|\xi|^2}\\
0&-\frac{\overline{\lambda}}{\lambda}&\frac{i}{\;\overline{\lambda}\;} \\
\lambda|\xi|^2&-i|\xi|^2\frac{\overline{\lambda}}{\lambda}&\frac{1}{2}\frac{-|\xi|^2+iu}{\overline{\lambda}}
\end{pmatrix}.$$

\noindent According to Lemma \ref{lemma eigenvalues}, the eigenvalue $e^{i\theta}$ is in $\Tr ({\Gamma})$. To see that the coefficients are the traces of some elements of ${\Gamma}$, we introduce the following function:

$$t(X)=(\Tr PX-e^{i\theta}\Tr X)(e^{-2i\theta}+e^{i\theta})-(\Tr P^2 X-e^{2i\theta} \Tr X).$$
\noindent A calculation gives: $t(L)=-\frac{e^{i3\theta}-1}{2e^{i\theta}}i\lambda$. If $e^{i\theta}\neq 1$, as $e^{i\theta}$ and $t(L)$ are in the trace field, so $i\lambda$ is in $\corps$. In the unipotent case, when $e^{i\theta}= 1$, then look at Tr($LP$)-Tr($L$)$=\frac{i\lambda}{2}$. Again, $i\lambda \in$ Tr$({\Gamma})$. Observe that Tr(L)=$-\frac{\overline{\lambda}}{\lambda}+\frac{1}{2}\frac{-|\xi|^2+iu}{\overline{\lambda}}$. As Tr$({\Gamma})$ is stable by complex conjugation, $i|\xi|^2$ and $u$ are in the trace field. So, the matrices $P$ and $L$ are in SU(2,1,Tr$({\Gamma})$).

\bigskip

\noindent \textbf{Second case:}

\bigskip

\noindent $P$ has the following form: $ e^{i\theta}\begin{pmatrix}
1&-e^{3i\theta}&\frac{-1+is}{2} \\
0&e^{-3i\theta}&1 \\
0&0&1
\end{pmatrix}$.
\\Lemma \ref{lemma eigenvalues} proves that the eigenvalue $e^{i\theta}$ of $P$ is in the trace field. In this case, we have: 
\\$t(L)=-\frac{1}{2}\lambda e^{-i\theta}(-1-e^{3i\theta}-is+ie^{3i\theta}s)$.
This is not zero. Indeed, if $t(L)=0$, then $s$ is equal to the real $i\dfrac{1+e^{3i\theta}}{1-e^{3i\theta}}$ and in this case, $P$ is elliptic. So, $\dfrac{t(L^{-1})}{t(L)}=\dfrac{\overline{\lambda}}{\lambda}$ is in the trace field. With Tr$(L)$ and Tr$(L^{-1})$, we prove that $\frac{|\xi|^2}{\lambda}$ and $i\frac{u}{\lambda}$ are in $\corps$. We observe that : 
\[4 e^{i\theta}\dfrac{t([P,L])}{t(L)}-\Tr((LP)^{-1})-e^{-i\theta}\Tr(L^{-1})=-e^{2i\theta}\dfrac{\ov{\lambda}^2}{\lambda^2}\ov{\xi}+(e^{2i\theta}-e^{-i\theta})\left(\dfrac{|\xi|^2}{\lambda}+\dfrac{\lambda}{\;\ov{\lambda}\;}\right)\in \corps.\]
So, $\xi\in \corps$ and it follows that $\lambda$ is in the field too. For the last coefficient $is$, see: \[\Tr(PL)-e^{i\theta}\Tr(L)=\frac{e^{i\theta}\lambda}{2}is-\frac{e^{i\theta}}{2}\lambda-\frac{\ov{\lambda}}{\lambda}(e^{i\theta}-e^{-2i\theta}+e^{i\theta}\ov{\xi}).\]
\noindent We can conclude that $is\in \corps$ like all the coefficients of $P$ and $L$.

\bigskip

As the group $\Gamma$ is Zariski dense, we can find in $\left\langle P,L \right\rangle $ a basis of $M_9(\C) : (\text{Id}, A_1,\cdots A_8)$ that completes the proof. 
\end{proof}

\begin{rem} The hypothesis Zariski dense is necessary: if we allow $\Gamma$ to fix one point, we can consider a unipotent parabolic group that fixes infinity. The trace field is $\mathbb{Q}$ and we cannot control the coefficients of the matrices. 
\\The group generated by $P= \begin{pmatrix}
1&0&\frac{i}{2} \\
0&1&0 \\
0&0&1
\end{pmatrix}$ and $L= \begin{pmatrix}
0&0&\frac{i}{2} \\
0&-1&0 \\
2i&0&1
\end{pmatrix}$ fixes a complex line. The trace field is $\mathbb{Q}$, and no conjugate group has their coefficients contained in $\mathbb{R}$. 
\end{rem}

\section{Invariant fields}
There is a more geometric description of the trace field. To see it, we introduce another field: the field of the invariants.

Let $M$ denote an oriented non-compact 3-dimensional manifold. Topologically, $M$ is the interior of a 3-manifold $\ov{M}$ whose boundary is the union of a finite collection of 2-dimensional tori. Consider the holonomy representation $\rho$ of $M$ into PU(2,1). Let $\Gamma$ denote the image group : $\rho(\Pi_1(M))$. According to GNA, $M$ admits an ideal triangulation. Each tetrahedron of the triangulation can be realized as a quadruple of points in $S^3$. Each face is identified with another one by an application in PU(2,1). Throughout this paper, we will suppose that the representation $\rho$ is compatible with the triangulation, that means $\Gamma$ is the group generated by the face-pairing. The representation can be either not discrete or not faithful. At least, it is finitely generated.

In \cite{Falbel}, the author introduces an invariant in $(\C-\{0,1\})^4$ that classifies quadruple of points in the sphere in generic position up to PU(2,1). (By generic position, we mean that no three of them are in a same $\C$-circle). The invariant of a quadruple $(p_0,p_1,p_2,p_3)$ is four cross-ratios given by the following contruction: we work in the unit ball model, in the section $z_3=1$ such that the hyperbolic plane is the ball $B\subset \C^2$ and $p_0,p_1,p_2,p_3$ are in $S^3$. The set of complex lines of $\C^2$ passing through $p_0$ is identified with $\C P^1$. The complex line passing trough $p_0$ and $p_i$ is represented in $\C P^1$ by $p'_i$ $(i=1,2,3)$. The tangent space of $S^3$ in $p_0$ intersects the ball centered at $p_0$ and that gives a fourth point $c'$ in $\C P^1$. We denote $z_1$ the cross-ratio of the four points $c',p'_1,p'_2,p'_3\in \C P^1$. With the same construction at the other vertices, one obtains the invariant of the tetrahedron $(z_1,z'_1,\tilde{z}_1,\tilde{z}'_1)\in \C^4$. The next proposition gives formulas for the invariant.

\begin{prop}\emph{\cite{Will}}
Let $\tilde{p}_i$ be a lift of $p_i$ and $c_{jk}$ a polar vector of the complex line passing through $p_j$ and $p_k$. Then, $(z_1,z'_1,\tilde{z}_1,\tilde{z}'_1)$ are given by the following formulas :

\begin{eqnarray}
z_1=\dfrac{\left\langle \tilde{p}_3, \tilde{c}_{01}\right\rangle \left\langle \tilde{p}_2,\tilde{p}_0 \right\rangle }{\left\langle \tilde{p}_2,\tilde{c}_{01} \right\rangle \left\langle \tilde{p}_3, \tilde{p}_0\right\rangle } =\left[p_0,c_{01},p_2,p_3 \right],
\label{eq_invariant} \qquad z'_1=\dfrac{\left\langle \tilde{p}_2,\tilde{c}_{01} \right\rangle \left\langle \tilde{p}_3, \tilde{p}_1 \right\rangle }{\left\langle \tilde{p}_3, \tilde{c}_{01}\right\rangle \left\langle \tilde{p}_2, \tilde{p}_1\right\rangle } =\left[c_{01},p_1,p_2,p_3 \right],  
\\ \tilde{z}_1=\dfrac{\left\langle \tilde{p}_1,\tilde{c}_{23} \right\rangle \left\langle \tilde{p}_0, \tilde{p}_2 \right\rangle }{\left\langle \tilde{p}_0, \tilde{c}_{23}\right\rangle \left\langle \tilde{p}_1, \tilde{p}_2\right\rangle } =\ov{\left[p_0,p_1,p_2,c_{23} \right]}, \qquad  \tilde{z}'_1=\dfrac{\left\langle \tilde{p}_0,\tilde{c}_{23} \right\rangle \left\langle \tilde{p}_1, \tilde{p}_3 \right\rangle }{\left\langle \tilde{p}_1, \tilde{c}_{23}\right\rangle \left\langle \tilde{p}_0, \tilde{p}_3\right\rangle } =\ov{\left[p_0,p_1,c_{23},p_3 \right]}.
\end{eqnarray}
\end{prop}

\begin{defn}
$k_\Delta$ denotes the field generated by the invariants of all the tetrahedra.
\end{defn}

We prove the following properties of the field $k_{\Delta}$.

\begin{prop}
\item \relax 
\begin{enumerate}
\item $k_\Delta$ is a commensurability invariant.

\item $k_\Delta$ is stable by complex conjugation.

\item Without loss of generality, we may assume that $0,\infty$ and $(1,t)$ are vertices of one of the tetrahedra. In that case, $k_\Delta$ is the field generated by the coordinates of the vertices. More precisely, if the normalized lifts of the points in $\C^3$ are $(w_{j1},w_{j2},1)$, then $k_\Delta=\mathbb{Q}(w_{j1}, \overline{w_{j1}},w_{j2}, \overline{w_{j2}} )$.
\end{enumerate}

\end{prop}

\begin{proof}
GNA If we have a finite cover of $M$, the new tetrahedra have the same invariants as the previous ones.

\noindent (i) The following formula gives $\ov{z_1}$:
 \[\ov{z_1}=
 \frac {( - z'_1 + z_1
z'_1 + \tilde{z}'_1 z'_1 - \tilde{z}'_1)\tilde{z}_1}{
\tilde{z}_1z_1 - \tilde{z}_1 + z_1z'_1 -z_1}. \]

For the others, we conclude in the same way.

\bigskip

\noindent (ii) The formulas that give $z_1,z'_1,\tilde{z}_1$ and $\tilde{z}'_1$ show that the invariants are calculated with the coordinates. So $k_\Delta$ is included in the field of the coordinates. For the second inclusion, we know that ${}^t(1,0,0)$, ${}^t(0,0,1)$, ${}^t(it,0,1)$ are the vertices of one tetrahedron. The number $it$ can be written with the invariants: 

\[ it=\dfrac{\tilde{z}_1(z_1-1)-z_1(z'_1-1)}{\tilde{z}_1(z_1-1)+z_1(z'_1-1)}.\] 
To conclude, we may consider that the coordinates of the fourth point of a tetrahedron can be calculated with the coordinates of the others three vertices and with the invariants. With (\ref{eq_invariant}):

\[ \left\langle \tilde{p}_3, \tilde{c}_{01}\right\rangle -  \left\langle \tilde{p}_3, \tilde{p}_0\right\rangle \dfrac{\left\langle \tilde{p}_2,\tilde{c}_{01} \right\rangle z_1}{\left\langle \tilde{p}_2,\tilde{p}_0 \right\rangle }= \left\langle \tilde{p}_3, \tilde{c}_{01}  -  \tilde{p}_0 \dfrac{\left\langle \tilde{c}_{01},\tilde{p}_2 \right\rangle \overline{z_1}}{\left\langle \tilde{p}_0,\tilde{p}_2 \right\rangle}\right\rangle =0 \]
\[  \left\langle \tilde{p}_3, \tilde{c}_{01}\right\rangle -  \left\langle \tilde{p}_3, \tilde{p}_1\right\rangle \dfrac{\left\langle \tilde{p}_2,\tilde{c}_{01} \right\rangle}{\left\langle \tilde{p}_2,\tilde{p}_1 \right\rangle z'_1} = \left\langle \tilde{p}_3, \tilde{c}_{01}  -  \tilde{p}_1 \dfrac{\left\langle \tilde{c}_{01},\tilde{p}_2 \right\rangle }{\left\langle \tilde{p}_1,\tilde{p}_2 \right\rangle \overline{z'_1}} \right\rangle  = 0 \]
So the coordinates of the fourth point $\tilde{p}_3$ are the solutions of a invertible linear system and so they are in $k_{\Delta}$. Gradually, we conclude that all the coordinates lie in $k_{\Delta}$.
\end{proof}

\begin{prop}
If each vertex of the triangulation is a parabolic fixed point and the holonomy $\Gamma$ is Zariski dense then $k_\Delta=k(\Gamma)$.
\end{prop}

\begin{proof}
The first inclusion $k(\Gamma) \subset k_\Delta$ is always true, even if $\Gamma$ does not contain a parabolic transformation and is not Zariski dense. It comes from the following lemma:

\begin{lemma}
Let $\gamma\in \Gamma$ and $M\in SU(2,1)$ a lift of $\gamma$. The non homogeneous coordinates of the images of $0, \infty, (1,t)$ by $\gamma$ are noted $\tilde{p}_0=\,^t(w_{01},w_{02},1), \tilde{p}_1=\,^t(w_{11},w_{12},1), \tilde{p}_2= \,^t(w_{21},w_{22},1)$. Let $K$ be the field stable by complex conjugation generated by the coordinates: $K=\Q(it,w_{jk},\overline{w_{jk}})$. Then  $M^3 \in$ SU$(2,1,K)$.
\end{lemma}

\begin{proof}
The following matrix sends $\infty$ to $p_0$ and $0$ to $p_1$ :
 $$M=\left(
\begin{array}{ c c c }
\mu w_{01}&  a  &\lambda w_{11} \\
\mu w_{02}& b   & \lambda w_{12}\\
\mu&  c  &\lambda
  \end{array} \right).
$$

\noindent $M$ is in SU(2,1) if it satisfies $^t\ov{M}JM=\text{Id}$, that induces the following equation for its coefficients:

\[a=\dfrac{\overline{\lambda}}{\lambda}\dfrac{\overline{w_{01}}\overline{w_{12}}-\overline{w_{02}}\overline{w_{11}} }{\left\langle \tilde{p}_1,\tilde{p}_0\right\rangle }, \quad b=\dfrac{\overline{\lambda}}{\lambda}\dfrac{\overline{w_{01}}-\overline{w_{11}}}{\left\langle \tilde{p}_1,\tilde{p}_0\right\rangle}, \quad c=\dfrac{\overline{\lambda}}{\lambda}\dfrac{\overline{w_{12}}-\overline{w_{02}}}{\left\langle \tilde{p}_1,\tilde{p}_0\right\rangle},  \quad \mu=\dfrac{1}{\overline{\lambda}\left\langle \tilde{p}_0,\tilde{p}_1\right\rangle}.\]

\[M=\lambda\left(
\begin{array}{ c c c }
\dfrac{w_{01}}{|\lambda|^2\left\langle \tilde{p}_0,\tilde{p}_1\right\rangle} &  \dfrac{|\lambda|^2}{\lambda^3}\dfrac{\overline{w_{01}}\overline{w_{12}}-\overline{w_{02}}\overline{w_{11}} }{\left\langle \tilde{p}_1,\tilde{p}_0\right\rangle }  & w_{11} 
\\ \dfrac{w_{02}}{|\lambda|^2\left\langle \tilde{p}_0,\tilde{p}_1\right\rangle}&  \dfrac{|\lambda|^2}{\lambda^3}\dfrac{\overline{w_{01}}-\overline{w_{11}}}{\left\langle \tilde{p}_1,\tilde{p}_0\right\rangle}  &  w_{12}
\\\dfrac{1}{|\lambda|^2\left\langle \tilde{p}_0,\tilde{p}_1\right\rangle} &  \dfrac{
|\lambda|^2}{\lambda^3}\dfrac{\overline{w_{12}}-\overline{w_{02}}}{\left\langle \tilde{p}_1,\tilde{p}_0\right\rangle}  &1
  \end{array} \right) .\]

\noindent The image of $(1,t)$ is $\tilde{p}_3$ : $M \,^t( \frac{-1}{2}+i\frac{t}{2} , 1 ,1 ) =\alpha \,^t( w_{20} ,w_{21},1 ).$ Eliminating $\alpha$, we find: 
\[|\lambda|^2=\dfrac{(\frac{-1}{2}+i\frac{t}{2})\left\langle \tilde{p}_2,\tilde{p}_0\right\rangle }{\left\langle \tilde{p}_2,\tilde{p}_1\right\rangle  \left\langle \tilde{p}_1,\tilde{p}_0\right\rangle},\hspace{1cm}\lambda^3=-\dfrac{(\frac{-1}{2}+i\frac{t}{2})\left\langle \tilde{p}_0,\tilde{p}_1\right\rangle \left\langle \tilde{p}_2,\tilde{p}_0\right\rangle^2}{\det(\tilde{p}_0,\tilde{p}_1,\tilde{p}_2)\left\langle \tilde{p}_2,\tilde{p}_1\right\rangle \left\langle \tilde{p}_1,\tilde{p}_0\right\rangle^2}.\]
Observe that $|\lambda|^2$ and $\lambda^3$ are in $K$, but $\lambda$ is not necessarily in the field. The coefficients of the matrix $\frac{1}{\lambda}M$ are in $K$. So we can not conclude for the coefficients of $M$. On the other hand, it becomes obvious that $M^3$ is in SU$(2,1,K)$.
\end{proof}

The set $\widetilde{\Gamma}$ is the group of the lifts in SU(2,1) of $\Gamma$. The previous lemma proves that the coefficients of the matrices in $\widetilde{\Gamma}^3$ are in $k_{\Delta}$. So we have the first inclusion.

For the second inclusion, we use Theorem \ref{prop coeff trace} so that $\widetilde{\Gamma}\subset \text{SU}(2,1,\text{Tr}(\widetilde{\Gamma})$). To calculate $k_\Delta$, we use the coordinates of the vertices. By hypothesis, they are fixed parabolic points. The following lemma proves that their coordinates are rational functions  of the coefficients of matrices in $\widetilde{\Gamma}$. So $k_\Delta \subset \text{Tr}(\widetilde{\Gamma})$. The solution of the Burnside problem implies that $\widetilde{\Gamma}^{(3)}$ is a finite index subgroup of $\widetilde{\Gamma}$. As the field $k_\Delta$ is a commensurability invariant, $k_\Delta \subset \text{Tr}(\widetilde{\Gamma}^3)=k(\Gamma)$.

\begin{lemma}
Let $P\in \text{SU}(2,1)$ be a lift of a parabolic transformation. The non homogeneous coordinates of its fixed point in $V_0$ are given by $\tilde{p}=\,^t(w_1,w_2,1)$. Then $w_1$ and $w_2$ are in the field generated by the coefficients of $P$.
\end{lemma}

\begin{proof}
The complex number $e^{i\theta}$ is the eigenvalue of $P$ associated to $\tilde{p}$. By Lemma \ref{lemma eigenvalues}, we know that it is in $\mathbb{Q}(\text{Tr} P,\overline{\text{Tr} P})$. The point $\tilde{p}$ satisfies the relation: $(P-e^{i\theta}\text{Id})\tilde{p}=0$. That gives a linear system in the variables $w_1$ and $w_2$ with coefficients in the field generated by the coefficients of $P$. It has a unique solution because a parabolic transformation has a unique fixed point in $S^3$. So the solutions $w_1$ and $w_2$ of this invertible system are rational functions of the coefficients of $P$.
\end{proof}
\end{proof}

\section{Arithmetic groups and trace fields}

We are looking for the same conclusions with arithmetic lattices. We deal only with arithmetic lattices of the first type, so they are the only ones we define. To see a complete description of arithmetic lattices in SU($n$,1), see \cite{McReynolds}.

\begin{defn}
 \label{defn arithmetic}
Let $E$ be a totally imaginary quadratic extension of a totally real number field $F$. The degree $[F:\Q]$ is noted $d$. Up to conjugacy, there are $d$ embeddings $\tau_i:E\rightarrow \C$ and associated to each of them $\sigma_i:F\rightarrow \R$. We may assume that $E\subset \C$ and $F=\R \cap E \subset \R$ and so $\sigma_1=\text{Id}$ and $\tau_1=\text{Id}$.

$H$ is a Hermitian matrix of signature $(n,1)$ and its coefficients are in $E$. We denote $^{\tau_i}H$ the Hermitian matrix obtained by replacing the coefficients of $H$ by their images under $\tau_i$. The pair $(H,E/F)$ is \textit{admissible} if SU$(^{\tau_i}H)$ is compact for all $i=2,\cdots d$.

The set $\mathcal{O}_E$ is the ring of integers of $E$. The group SU$(H,\mathcal{O}_E)$ is the set of matrices that preserve $H$ whose coefficients are in $\mathcal{O}_E$. It is the fundamental example of arithmetic lattice when $(H,E/F)$ is admissible.

A group $\Gamma\subset \text{SU}(n,1)$ is an \textit{arithmetic lattice of first type} if it is conjugated to a group commensurable with $\text{SU}(H,\mathcal{O}_E)$ for an admissible couple $(H,E/F)$.
\end{defn}

\begin{lemma} \label{lemme 1}Let $\Gamma$ be a subgroup of SU$(H,\mathcal{O}_E)$. If $\Gamma$ contains a loxodromic transformation, then $\corps= E$ or $\corps= F$.
\end{lemma}
\begin{proof}
Obviously, we have $\corps\subset E$. If $\corps \neq E$, as $[E:F]=2$, then $\corps \subset F$. Suppose the two fields are not equal. Then we can find an embedding  $\sigma : F\rightarrow \R$ which is not trivial such that $\sigma_{|\corps}=\text{Id}$. The application $\sigma$ is the restriction to $F$ of a place $\tau\colon E\rightarrow \C$ which is neither identity nor complex congugation. 

We denote $\gamma$ the loxodromic transformation, $\lambda$ the eigenvalue of $\gamma$ whose modulus is not 1 and $e^{i\theta_j}$ the other eigenvalues. Then:
\[\Tr \:\gamma^m=\lambda^m +\dfrac{1}{\overline{\lambda}^m}+e^{i m\theta_1 }+\cdots+ e^{i m\theta_{n-1} } .\]
So, $\lim_{m\rightarrow\infty}\Tr \:\gamma^m=\infty$. On the other hand, $\text{Tr}(\gamma)=\sigma(\text{Tr}(\gamma))=\tau(\text{Tr}(\gamma))=\text{Tr}(^{\tau} \gamma)$. The group $\Gamma$ is included in the arithmetic group SU$(H,\mathcal{O}_E)$. So $^\tau \gamma$ is in the compact group SU$(^\tau H)$ and the trace function is bounded. That contradiction proves the result.
\end{proof}

\begin{prop}
\label{prop2}
Let $\Gamma$ be an arithmetic lattice of the first type in SU(n,1). Then, up to conjugacy, the coefficients of $\Gamma$ are in $\corps (\alpha)$ \emph{(}where $E=F(\alpha)$\emph{)}.
\end{prop}

\begin{proof}
After conjugation, we may assume that $\Gamma$ is commensurable with SU$(H,\mathcal{O}_E)$ and we denote $\Gamma_1$ a group that is a finite index subgroup of both of them. As an arithmetic lattice, $\Gamma_1$ contains a loxodromic transformation. From lemma \ref{lemme 1}, $\text{Tr }(\Gamma_1)=E$ or $\text{Tr }(\Gamma_1)=F$. 

As in proposition \ref{preuve de Mc}, the Skolem-Noether theorem implies that for all $g\in \Gamma$, there exist $g_1\in A(\Gamma_1)$ and $\beta\in\mathbb{C}^*$ such that $g=\beta g_1$. The coefficients of $g_1$ are in $E$ itself included in $\text{Tr}(\Gamma_1)(\alpha)\subset \corps(\alpha)$. What remains to be shown is that $\beta\in \corps(\alpha)$. As $\Tr g =\beta \Tr g_1$, if $\Tr g \neq 0$, that completes the proof. Otherwise, take a loxodromic element $\gamma$. For $n$ large enough, $\Tr \gamma^n$ and $\Tr \gamma^{-n}g$ are not zero. By the previous argument, $  \gamma^n$ and $\gamma^{-n}g$ have their coefficients in $\corps(\alpha)$, and so $g$ too.
\end{proof}

If $\Gamma$ is an arithmetic lattice contained in SU$(H,\mathcal{O}_E)$, then $\corps\subset E$ as we have seen in lemma \ref{lemme 1}. This is generally not true as is proved by the following counterexample: in SU(1,1), we construct an arithmetic group such that $E\subsetneq\corps$.

\begin{ex}
Let $E=\Q(\sqrt{15},i)$, $F=\Q(\sqrt{15})$ and $H=\text{diag}(-\sqrt{15},1)$. We denote by $\Gamma$ the arithmetic lattice SU$(H,\mathcal{O}_E)$ and $g=\sqrt{4+\sqrt{15}}\begin{pmatrix}
2&1 \\
\sqrt{15}&2 
\end{pmatrix}$. The element $g$ satisfies: $g^*Hg=H$, det$(g)=1$ and $g^2 \in \Gamma$. The group $\Gamma$ has index two in $\widetilde{\Gamma}=\left\langle \Gamma,g\right\rangle $ and so $\widetilde{\Gamma}$ is arithmetic. The trace field of this group contains $\sqrt{4+\sqrt{15}}$ which is not in $E$.
\end{ex}

\begin{rem}
In the real case, some arithmetic groups are constructed following definition \ref{defn arithmetic}. Let $k$ be a totally real number field and $\tau_i$ the real embeddings of $k$. In SO($n$,1), the matrix $h$ is symmetric of signature $(n,1)$ so that SO$(^{\tau_i}h)$ is compact. The group SO$(h,\mathcal{O}_k)$ is an arithmetic lattice. In SO($2m,1)$, we obtain all the arithmetic lattices in that way (up to commensurability and conjugation). An application of the Skolem-Noether theorem proves that these lattices satisfy $k=\corps$. 
Again, as in the complex case, this is generally false when $n$ is odd.
\end{rem}

\end{document}